\documentclass[12pt]{amsart}

\usepackage{amsfonts}
\usepackage{amsmath}
\usepackage{amssymb}
\usepackage{amscd}
\usepackage{mathrsfs}  % This package allows the use of script letter. Type \mathscr to invoke math script.
\usepackage{amsbsy}
\usepackage{amsthm}
\usepackage{graphicx}
\usepackage{fancyhdr}
\usepackage{epsfig}
\usepackage{color}
\usepackage{xy}
\usepackage{enumitem}
\usepackage{bbm}
\usepackage{upgreek,textgreek}
\usepackage{hyperref}

\usepackage[OT2,OT1]{fontenc}  % This package allows the use of cyrillic fonts. They are invoked by typing \textcyr{...}
\newcommand\cyr{%
\renewcommand\rmdefault{wncyr}%
\renewcommand\cFdefault{wncyss}%
\renewcommand\encodingdefault{OT2}%
\normalfont \selectfont} \DeclareTextFontCommand{\textcyr}{\cyr}

\newcommand{\be}{\begin{equation}}
\newcommand{\ee}{\end{equation}}

\newcommand{\bes}{\begin{equation*}}
\newcommand{\ees}{\end{equation*}}

\newcommand{\N}{\mathbb{N}}

\newcommand{\R}{\mathbb{R}}

\newcommand{\X}{\mathbb{X}}

\newcommand{\Z}{\mathbb{Z}}

\newcommand{\cB}{\mathcal{B}}

\newcommand{\cF}{\mathcal{F}}

\newcommand{\cW}{\mathcal{W}}

\newcommand{\sF}{\mathscr{F}}

\newcommand{\sA}{\mathscr{A}}

\newcommand{\sfB}{\mathsf{B}}
\newcommand{\sfE}{\mathsf{E}}
\newcommand{\sfF}{\mathsf{F}}

\newcommand{\sfP}{\mathsf{P}}
\newcommand{\sfU}{\mathsf{U}}
\newcommand{\sfV}{\mathsf{V}}
\newcommand{\sfX}{\mathsf{X}}
\newcommand{\sfY}{\mathsf{Y}}

\newcommand{\cC}{\mathcal{C}}

\newcommand{\fF}{\mathfrak{F}}

\newcommand{\nl}{\vskip 10pt\noindent}

\newcommand{\oU}{\overline{\sfU}}

\renewcommand{\rmdefault}{cmr} % Arial
\newtheorem{theorem}{Theorem}[section]
\theoremstyle{plain}

\newtheorem{definition}{Definition}[section]
\newtheorem{example}{Example}[section]

\newtheorem{lemma}{Lemma}[section]

\newtheorem{proposition}{Proposition}[section]
\newtheorem{remark}{Remark}[section]
\newtheorem{remarks}{Remarks}[section]

\newtheorem*{notation*}{Notation}
\numberwithin{equation}{section}

\DeclareMathOperator\Lip{Lip}

\DeclareMathOperator*\esssup{{ess\, sup}}

\DeclareMathOperator\gr{graph}

\DeclareMathOperator\Diff{Diff}

\newcommand{\sH}{\mathscr{H}}

\newcommand{\n}[1]{{\left\|{#1}\right\|}}
\newcommand{\abs}[1]{\left\vert{#1}\right\vert}
\newcommand{\floor}[1]{\left\lfloor{#1}\right\rfloor}
\newcommand{\vertiii}[1]{{\left\vert\kern-0.25ex\left\vert\kern-0.25ex\left\vert #1 
    \right\vert\kern-0.25ex\right\vert\kern-0.25ex\right\vert}}
\newcommand{\vep}{\varepsilon}
\newcommand{\one}{1\hspace{-0.23em}\mathrm{l}}

\begin{document}

\title[Fractal Interpolation over Nonlinear Partitions]{Fractal Interpolation over \\Nonlinear Partitions}
\author{Peter R. Massopust}
\address{Centre of Mathematics, Technical University of Munich, Boltzmannstrasse 3, 85747
Garching b. M\"unchen, Germany}
\email{massopust@ma.tum.de}

\begin{abstract}
This paper introduces the fractal interpolation problem defined over domains with a  nonlinear partition. This setting generalizes known methodologies regarding fractal functions and provides a new holistic approach to fractal interpolation. In this context, perturbations of nonlinear partition functions are considered and sufficient conditions for the existence of a unique solution of the underlying fractal interpolation problem for some classes of function spaces are given.
\end{abstract}

\maketitle
\section{Introduction}
We introduce the following notation and terminology: $\N$ denotes the set of positive natural numbers, $\N_0 := \N\cup\{0\}$, and $\N_n := \{1, \ldots, n\}$ the initial segment of $\N$ of length $n$.

This paper appertains to the Fractal Interpolation Problem: Given a bounded subset  
$\sfX$ of a Banach space $\sfE$ and a Banach space $\sfF$, construct a global function $\psi:\sfX = \coprod\limits_{i=1}^n \sfX_i\to\sfF$ belonging to some prescribed function space $\sF:= \sF(\sfX,\sfF)$ satisfying $n$ functional equations of the form
\[
\psi (h_i (x)) = q_i (x) + s_i (x) \psi (x), \quad\text{on $\sfX$ and for $i\in \N_n$},
\]
where the functions $h_i$ partition $\sfX$ into disjoint subsets $\sfX_i = h_i(\sfX)$, $q_i\in \sF$, and the functions $s_i$ are chosen so that the product $s_i\cdot\psi\in\sF$. Note that, if such a global solution exists, then it is pieced together in a prescribed manner from copies of itself on the subsets $\sfX_i$. This latter property defines the self-referential nature of $\psi$ and expresses the fact that the graph of $\psi$ is in general a fractal set, i.e., an object of immense geometric complexity.

The fractal interpolation problem has been addressed in numerous publications and for different function spaces $\sF$ where the domains of $f\in \sF$ were one-dimensional, $n$-dimensional or infinite-dimensional. In a large number of these publications and for the fractal constructions upon which they are based, in particular in the one-dimensional setting, the partitions of the domain $\sfX$ are assumed to be induced by linear or affine functions. In this paper, we explicitly introduce nonlinear partition functions $h_i$ of $\sfX$ and show how nonlinear partitions induced by  such functions provide for more flexibility in the fractal interpolation problem.

This paper is organized as follows. In Section 2, iterated function systems are defined and related issues briefly discussed. The next section introduces the novel concept of fractal interpolation over nonlinear partitions and proves the existence of bounded and continuous solutions. In this context, perturbations of the nonlinear partition functions $h_i$ are also considered and the connection between the attractor of an iterated function system and the solution to the fractal interpolation problem is exhibited. Sections 4 and 5 deal with the existence of solutions of the fractal interpolation problem over nonlinear partitions in Bochner-Lebesgue and $\cC^\alpha$ function spaces.
\section{Iterated Functions Systems}
In the following, $\sfE:=(\sfE, \n{\cdot}_\sfE)$ and $\sfF:=(\sfF, \n{\cdot}_\sfF)$ always denote Banach spaces and the subscript on the norm is dropped if no confusion is to be expected.

For a map $f: \sfE \to \sfE$, we define the Lipschitz constant associated with $f$ by
\[
\Lip (f) = \sup_{x\in\sfE, x \neq 0} \frac{\n{f(x)}}{\n{x}}.
\]
A map $f$ is called \emph{Lipschitz} if $\Lip (f) < + \infty$ and a \emph{contraction} if $\Lip (f) < 1$.

\begin{definition}
Let $\sfE$ be a Banach space and $\cF$ a finite set of functions $\sfE\to\sfE$. Then, the pair $(\sfE,\cF)$ is called an iterated function system (IFS) on $\sfE$. In case all maps $f\in\cF$ are contractions then $(\sfE,\cF)$ is called a \emph{contractive} IFS.
\end{definition}

For a more general definition of IFS and related topics, we refer the interested reader to \cite{BWL}.

\begin{remark}
Throughout this paper, we deal with contractive IFSs and therefore we drop the adjective ``contractive."
\end{remark}

With the finite set of contractions $\cF$ on $\sfE$, one associates a set-valued operator, again denoted by $\cF$, acting on the hyperspace ${\sH}(\sfE)$ of nonempty compact subsets of $\sfE$:
\[
\cF (E) := \bigcup_{f\in \cF} f (E),\qquad E\in \sH(\sfE).
\]
The hyperspace ${\sH}(\sfE)$ endowed with the Hausdorff-Pompeiu metric $d_{\sH}$ defined by
\[
d_{\sH} (S_1, S_2) := \max\{d(S_1, S_2), d(S_2, S_1)\},
\]
where $d(S_1,S_2) := \sup\limits_{x\in S_1} d(x, S_2) := \sup\limits_{x\in S_1}\inf\limits_{y\in S_2} \n{x-y}_\sfE$, becomes a metric space.

It is a known fact that the completeness of $\sfE$ implies the completeness of $({\sH}(\sfE), d_{\sH})$ as a metric space. Moreover, the set-valued operator $\cF$ is contractive on the complete metric space $({\sH}(\sfE), d_{\sH})$ with Lipschitz constant $\Lip\cF = \max\{\Lip (f) : f\in \cF\} < 1$ if all $f\in \cF$ are contractions. 

In this case and by the Banach Fixed Point Theorem, $\cF$ has a unique fixed point in $\sH(\sfE)$. This fixed point is called the \emph{attractor} of or the \emph{fractal (set)} generated by the IFS $(\sfE,\cF)$. The attractor or fractal $F$ satisfies the self-referential equation
\be\label{fixedpoint}
F = \cF (F) = \bigcup_{f\in\cF} f (F),
\ee
i.e., $F$ is made up of a finite number of images of itself. Eqn. \eqref{fixedpoint} reflects the fractal nature of $F$ showing that it is as an object of immense geometric complexity.

The proof of the Banach Fixed Point Theorem also shows that the fractal $F$ can be obtained iteratively via the following procedure: Choose an arbitrary $F_0\in {\sH}(\sfE)$ and set
\be\label{F}
F_n := \cF (F_{n-1}),\qquad n\in\N.
\ee
Then $F =\lim\limits_{n\to\infty} F_n$, where the limit is taken with respect to the Hausdorff-Pompeiu metric $d_{\sH}$.

%In Figure \ref{fig1}, two examples of fractal sets are displayed.
%\begin{figure}
%\begin{center}
%\includegraphics[width= 3cm, height = 2.5cm]{fractal.pdf}\hspace{2cm}
%\includegraphics[width= 3cm, height = 2.25cm]{ff.pdf}
%\caption{Left: A fractal set in $X:= [0,1]^2$ generated by the maps $f_1 (x,y) := (\frac12 x, \frac12 y)$, $f_2 (x,y) := (\frac12 (x+1), \frac12 y)$, and $f_3 (x,y) := (\frac12 (x+\frac12), \frac34 y + \frac{\sqrt{3}}{4})$. Right: The graph of a \emph{fractal function} in $X:= [0,1]\times[0,3]$ generated by the maps $f_1 (x,y) := (\frac12 x, x + \frac34 y)$ and $f_2 (x,y) := (\frac12 (x+1), x^2 + \frac34 y)$.}\label{fig1}
%\end{center}
%\end{figure}
For more details about IFSs and fractals and their properties, we refer the interested reader to the large literature on these topics and list only two references \cite{barnsley,massopust1} which are closely related to the present exhibition.

\section{Fractal Interpolation over Nonlinear Partitions}\label{sect3}

We briefly recall the rudimentaries of fractal interpolation and fractal functions. 
Let $\sfU$ and $\sfV$ be open subsets of the Banach spaces $\sfE$ and $\sfF$, respectively. A mapping $f:\sfU\to\sfV$ is called a \emph{$C^1$-diffeomorphism} if $f$ is a bijection, Fr\'echet differentiable on $\sfU$, and the inverse $f^{-1}$ is Fr\'echet differentiable on $\sfV$. In an analogous fashion, higher order diffeomorphisms $\sfU\to\sfV$ are defined. The collection of all $C^\alpha$-diffeomorphisms from $\sfU\to\sfV$ with $\alpha\in \N_0\cup\{\infty\}$ is denoted by $\Diff^\alpha (\sfU,\sfV)$. In case, $\sfU := \sfE := \sfF$, we simply write $\Diff^\alpha (\sfE)$. 

If $\sfX\subset\sfE$ then $f:\sfX\to\sfF$ is called a $C^\alpha$-diffeomorphism on $\sfX$, in symbols $f\in \Diff^\alpha(\sfX,\sfF)$, if there exists an open $\sfU\subset\sfE$ with $\sfX\subset\sfU$ and a $C^\alpha$-diffeomorphism $g:\sfU\to\sfF$ such that $f = g\lvert_\sfX$. 

\subsection{Fractal Interpolation}
Let $\sfX$ be a nonempty bounded subset of a Banach space $\sfE$. Suppose we are given a finite family $h:= \{h_i\in \Diff^\alpha (\sfX)  : i = 1, \ldots, n\}$ of $C^\alpha$-diffeomorphisms generating a partition $\Pi(h)$ of $\sfX$ in the sense that
\be\label{c1}
\sfX = \coprod_{i=1}^n h_i(\sfX),
%&h_i(\sfX)\cap h_j(\sfX) = \emptyset, \quad\forall\;i, j\in \N_n, i\neq j.\label{part2} 
\ee
where $\coprod$ denotes the disjoint union of sets. We will also set $\sfX_i := h_i(\sfX)$.\\

Recall that a mapping $f:\sfE\to\sfF$ is called \emph{affine} if $f - f(0)$ is linear.

\begin{definition}
A partition $\Pi(h)$ of $\sfX$ is called \emph{nonlinear} if the maps generating $\Pi(h)$ are not affine.
\end{definition}

In the current context, we may express \eqref{c1} as follows: The IFS $(\sfX, h)$ has as its attractor $\sfX$ which consists of finitely many disjoint \emph{nonlinear} images of itself.

\begin{example}\label{ex2.1}
A fairly straight-forward example is given by choosing $\sfE := \R$, $\sfX := [0,1)$, and
\[
h_1 (x):= \tfrac16(2x + x^2)\quad\text{and}\quad h_2(x):= \tfrac12(1 + \sqrt{2}\sin\tfrac{\pi x}{4}).
\]
Here, $h_1, h_2\in\Diff^\infty [0,1)$. The nonlinear partition $\Pi (h)$ of $\sfX$ generated by these two maps is then $\{[0,\frac12) = h_1(\sfX), [\frac12, 1) = h_2(\sfX)\}$. 
\end{example}

One of the goals of fractal interpolation is the construction of a global function $\psi:\sfX = \coprod\limits_{i=1}^n \sfX_i\to\sfF$ belonging to some prescribed function space $\sF$ and satisfying $n$ functional equations of the form
\be\label{psieq}
\psi (h_i (x)) = q_i (x) + s_i (x) \psi (x), \quad\text{on $\sfX$ and for $i\in \N_n$},
\ee
where for each $i\in\N_n$, $q_i\in\sF$ and $s_i$ is chosen such that $s_i\cdot\psi\in\sF$. In other words, the global solution ist pieced together in a prescribed manner from copies of itself on the subsets $\sfX_i = h_i(\sfX)$. 

We refer to \eqref{psieq} as the \emph{fractal interpolation problem}. To solve this problem, the idea is to consider \eqref{psieq} as the fixed point equation for an associated affine operator acting on an appropriately defined or prescribed function space.
\subsection{Bounded fractal functions}
To this end, recall that a mapping $f:\sfE\to \sfF$ is called \emph{bounded} if there exists an $M> 0$ such  that $\n{f(x)} < M$ for all $x\in \sfE$. 

Let $\cB (\sfX,\sfF) := \{f:\sfX\to\sfF : \text{$f$ is bounded}\}$ denote the Banach space of bounded functions equipped with the supremums norm $\n{f} := \sup\limits_{x\in \sfX} \n{f(x)}_\sfF$. \\

Suppose that $q_i\in \cB (\sfX,\sfF)$ and $s_i\in \cB(\sfX) :=\cB (\sfX,\R)$, $i\in\N_n$. On the Banach space $\cB (\sfX,\sfF)$, we define an affine operator $T: \cB (\sfX,\sfF)\to \cB (\sfX,\sfF)$, called a Read-Bajractarevi\'c (RB) operator, by 
\be\label{eq3.17}
T f (x) := (q_i\circ h_i^{-1})(x) + (s_i\circ h_i^{-1})(x)\cdot (f\circ h_i^{-1})(x), 
\ee 
for $x\in \sfX_i$ and $i\in \N_n$, or, equivalently, by
\begin{align*}
T f (x) &= \sum_{i=1}^n (q_i\circ h_i^{-1})(x)\, \one_{\sfX_i}(x) + \sum_{i=1}^n (s_i\circ h_i^{-1})(x)\cdot (f\circ h_i^{-1})(x)\, \one_{\sfX_i}(x)\\
&= T(0) + \sum_{i=1}^n (s_i\circ h_i^{-1})(x)\cdot (f\circ h_i^{-1})(x)\, \one_{\sfX_i}(x),\quad x\in \sfX,
\end{align*}
where $\one_{\sfX_i}: \sfX\to \{0,1\}$ denotes the characteristic or indicator function of the set $\sfX_i$ and ``$\cdot$" pointwise multiplication.

The following result is well-known (see, for instance, \cite{B2,massopust1}) but for the sake of completeness we reproduce the proof. We also refer the interested reader to \cite{SB} where a similar set-up is considered.

\begin{theorem}\label{sol}
The system of functional equations \eqref{psieq} has a unique bounded solution $\psi: \sfX\to\sfF$ provided that
\begin{enumerate}
\item $\sfX = \coprod\limits_{i=1}^n h_i(\sfX)$,
\item $q_i\in \cB (\sfX,\sfF)$ and $s_i\in \cB (\sfX,\R)$, $i\in\N_n$, and
\item $s:= \max\limits_{i\in \N_n} \sup\limits_{x\in \sfX} |s_i(x)| < 1$.
\end{enumerate} 
\end{theorem}

\begin{proof}
First note that, since the mappings $h_i$ are injective, the right-hand side of \eqref{psieq} can be written as the right-hand side of \eqref{eq3.17}.

As the functions $h_i$, $q_i$, and $s_i$ are all assumed to be bounded, $T$ maps $\cB (\sfX,\sfF)$ into itself. For all $f,g\in \cB (\sfX,\sfF)$, we have that
\begin{align}
\sup_{x\in\sfX}\n{Tf(x) - Tg(x)}_\sfF &= \max_{i\in \N_n}\sup_{x\in\sfX_i} \n{(s_i\circ h_i^{-1})(x)\cdot (f-g)\circ h_i^{-1}(x)}_\sfF\nonumber\\
&= \max_{i\in \N_n}\sup_{\xi\in\sfX} \n{s_i(\xi)\cdot (f-g)(\xi)}_\sfF\label{3.4}\\
&\leq \max_{i\in \N_n}\sup\limits_{x\in \sfX} |s_i(x)| \sup_{x\in\sfX}\n{(f-g)(x)}_\sfF,\label{3.5}
\end{align}
from which it follows that
\[
\n{Tf-Tg} \leq s \n{f-g}.
\]
Hence, $T$ is contractive on the Banach space $\cB (\sfX,\sfF)$ and therefore, by the Banach Fixed Point Theorem, has a unique fixed point $\psi\in\cB (\sfX,\sfF)$. This fixed point solves the functional equations \eqref{psieq}.
\end{proof}
\begin{example}\label{ex3.2}
Consider the set-up in Example \ref{ex2.1} with $\sfF := \R$. Suppose $q_1 (x) := -1$, $q_2(x) := \floor{\sqrt{10 x}}$, $s_1 (x) := -\frac12\sin{x}$, and $s_2(x) := -\frac23\cos{x}$. Here, $\floor{\cdot} :\R\to\Z$ denotes the floor function.

Define an RB operator (to solve the fractal interpolation problem) by
\[
Tf (x) := \begin{cases}
(q_1\circ h_1^{-1}) (x)  -\tfrac12(\sin\circ h_1^{-1})(x)\cdot (f\circ h_1^{-1}) (x), & x\in [0,\tfrac13),\\
(q_2\circ h_2^{-1}) (x)  -\tfrac23(\cos\circ h_2^{-1})(x)\cdot (f\circ h_2^{-1}) (x), & x\in [\tfrac13, 1).
\end{cases}
\]
Then, $q_1, q_2,s_1,s_2\in \cB[0,1)$ and $s = \frac23 < 1$. The unique solution of the fractal interpolation problem, i.e., the fixed point $\psi$ of the above RB operator, is shown in Figure \ref{fig0}.
\begin{figure}[h!]
\begin{center}
\includegraphics[width= 6cm, height = 4cm]{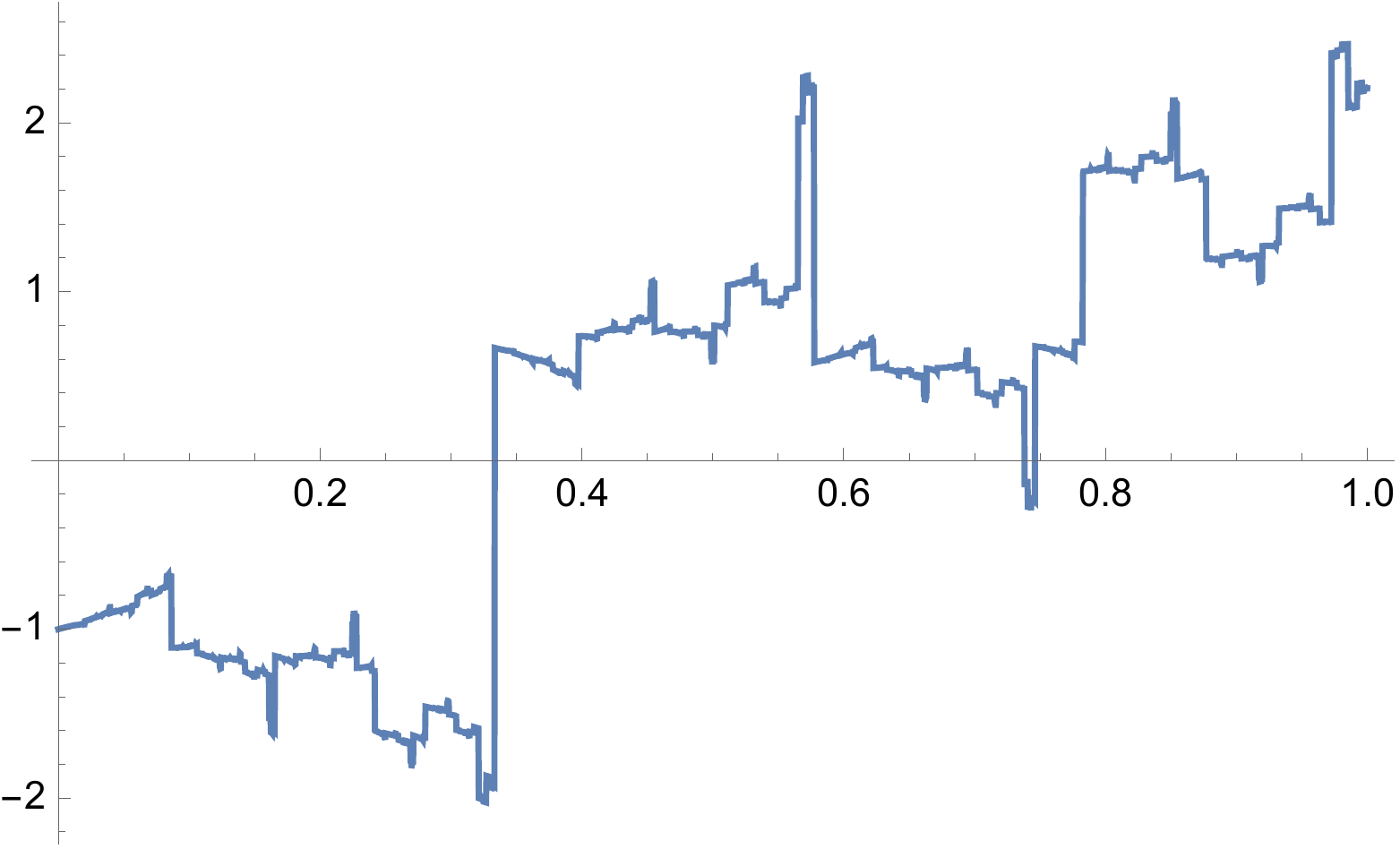}
\caption{The graph of a bounded fractal function $\psi$.}\label{fig0}
\end{center}
\end{figure}
\end{example}

At this point, several remarks are in order.
\begin{remarks}\label{rems3.1}\hfill
\begin{enumerate}
\item The fixed point $\psi\in\cB (\sfX,\sfF)$ of the RB operator $T$ is also called a \emph{bounded fractal function}. In this context, Eqn. \eqref{psieq} is also referred to as a \emph{self-referential equation} for $\psi$. 
\item The self-referential equation $T\psi = \psi$ expresses the fractal nature of the $\gr\psi$: It is made up of a finite number of copies of itself with each copy being supported on the partitioning sets $\sfX_i$. Hence, the terminology \emph{fractal function} for $\psi$.
\item The proof of Banach's Fixed Point Theorem also provides an algorithm for the construction of $\psi$: Choose \emph{any} function $\psi_0\in \cB (\sfX,\sfF)$ and iteratively define the following sequence of functions: 
$$
\psi_k := T \psi_{k-1}, \quad k\in \N. 
$$
Then, $\psi$ is given by $\psi = \lim\limits_{k\to\infty}\psi_k$ where the limit is taking with respect to the norm $\n{\cdot}$ on  $\cB (\sfX,\sfF)$.
\item The afore-mentioned algorithm for the construction of $\psi$ together with the proof of the Banach Fixed Point Theorem gives an error estimate as well, namely,
\[
\n{\psi - \psi_{k}} \leq \frac{s^k}{1-s} \n{\psi_1 - \psi_0}, \quad k\in \N.
\]
\item The fixed point $\psi$ depends on $n$, the partition $(\sfX_i : i\in\N_n)$, and the functions $s_i$ and $q_i$ with different choices yielding different fractal functions.
\item Emphasizing the dependence of $\psi$ on the functions $s_i$, the expression $s$-fractal function can be found in the literature. (See, for instance, \cite{N}.) In this context, one considers a fractal function as the image under an operator $\fF^s$ associating with a given (non-fractal) function $f$ its ``fractalization" $\psi_f = \fF^s f$ where $s = (s_1, \ldots, s_n)$.
\item Functional equations such as \eqref{psieq} exhibit connections to so-called \emph{fractels} \cite{bhm2,massopust1} and also to the approximation of rough functions \cite{BHVV}.
\item \eqref{psieq} is not the most general form of a functional equation that guarantees the existence of a fixed point $\psi$, that is, a fractal function. More generally, one can  consider mappings $v_i:\sfX\times\sfF\to\sfF$ which are uniformly contractive in the second variable, i.e., where there exists a $c\in [0,1)$ so that for all $y_1, y_2\in \sfF$
\be\label{scon}
\n{v_i (x, y)}_\sfF \leq c\, \n{y}_\sfF, \quad\forall x\in \sfX,\,\forall i \in\N_n.
\ee
Theorem \ref{sol} remains valid in this setting. We leave the easy to verify details to the interested reader. 
\end{enumerate}
\end{remarks}
\subsection{Connection between IFSs and the fractal interpolation problem}
Here, we exhibit the relation between the graph $G(\psi)$ of the fixed point $\psi$ of the operator $T$ and the attractor of the associated contractive IFS. We will do this for the more general setting mentioned above in Remarks \ref{rems3.1}(8).

For this purpose, consider the Banach space $\sfE\times\sfF$ with the product norm $\n{\cdot}_\sfE + \n{\cdot}_\sfF$. Let $\sfX\in\sH(\sfE)$ and let $\sfY\in\sH(\sfF)$ be such that the mappings $w_i:\sfX\times\sfY\to\sfX\times\sfF$ defined by
\be\label{wn}
w_i (x, y) := (h_i (x), v_i (x,y)), \quad i\in\N_n,
\ee
map into $\sfX\times\sfY$. (Cf. \cite[Proposition 27]{massopust1}.) 

Note that here the $h_i$ are defined on a compact subsets of $\sfE$ and in general $h_i(\sfX)\cap h_j(\sfX)\neq\emptyset$ for $i,j\in \N_n$ with $i\neq j$. However, we insist that 
$\sfX$ is \emph{non-overlapping} with respect to the IFS $(\sfX, h)$ in the sense of \cite[Definition 2.4]{bbhv} and that at the contact points (see Remark \ref{rem3.1}) the compatibility conditions \eqref{joinupc} hold. This will guarantee the existence of a solution $\psi$ \cite{SB}.

Assume that the mappings $v_i$ in addition to being uniformly contractive in the second variable are also uniformly Lipschitz continuous in the first variable, i.e., that there exists a constant $\lambda > 0$ so that for all $y\in\sfY$,
\[
\n{v_i(x_1, y) - v_i(x_2, y)}_\sfF \leq \lambda\, \n{x_1 - x_2}_\sfE, \quad\forall x_1, x_2\in \sfX,\quad\forall i \in\N_n.
\]
Let $\Lip (h) := \max\{\Lip{h_i} : i\in\N_n\}$ and let $\vartheta := \frac{1-\Lip(h)}{2\lambda}$. Then the mapping $\n{\cdot}_\vartheta : \sfE\times\sfF\to \R$ given by
\[
\n{\cdot}_\vartheta := \n{\cdot}_\sfE + \vartheta\,\n{\cdot}_\sfF
\]
is a norm on $\sfE\times\sfF$ compatible with the product topology on $\sfE\times \sfF$.

The next theorem is a special case of a result presented in \cite{bhm}.

\begin{theorem}\label{thm3.1}
Let $\cW := \{w_1, \ldots, w_n\}$. Then, $\cF := (\sfX\times\sfY, \cW)$ is a contractive IFS with respect to the norm $\n{\cdot}_\vartheta$ and the graph $G(\psi)$ of the solution to the fractal interpolation problem
\be\label{eq3.8}
\psi\circ h_i = v_i (x, \psi),\;\text{on $\sfX$ and $i\in\N_n$},
\ee
%
%fractal function $\psi$ generated by the RB operator 
%\be\label{eq3.8}
%T: \cB(\sfX,\sfF)\to\cB(\sfX,\sfF), \quad f\mapsto v_i (h_i^{-1}, f(h_i^{-1})),\;\;\text{on $\sfX_i$},\; i\in\N_n,
%\ee
is the unique attractor of the IFS $\cF$. 

Furthermore, if $T$ is the RB operator 
\be\label{eq3.9a}
T: \cB(\sfX,\sfY)\to\cB(\sfX,\sfY), \quad f\mapsto v_i (h_i^{-1}, f(h_i^{-1})),\;\;\text{on $\sfX_i$},\; i\in\N_n,
\ee
associated with the fractal interpolation problem \eqref{eq3.8} then
\be\label{GW}
G(T f) = \cF (G(f)),\quad\forall\,f\in \cB(\sfX,\sfY),
\ee
where $\cF$ denotes the set-valued operator \eqref{fixedpoint}.
\end{theorem}

Equation \eqref{GW} can be visualized by the commutative diagram
\be\label{diagram}
\begin{CD}
\sfX\times \sfY @>\cF>> \sfX\times \sfY\\
@AAGA                  @AAGA\\
\cB(\sfX,\sfY) @>T>>  \cB(\sfX,\sfY)
\end{CD}
\ee
\nl
where $G$ is the mapping $\cB(\sfX,\sfF)\ni g\mapsto G(g) = \{(x, g(x)) : x\in \sfX\}\in \sfX\times \sfY$.

On the other hand, assume that $\cF = (\sfX\times\sfY, w_1, w_2, \ldots, w_n)$ is an IFS whose mappings $w_i$ are of the form \eqref{wn} where the functions $h_i$ are nonlinear partition functions $h_i:\sfX\to\sfX_i$ with non-overlapping attractor $\sfX$ and contact points \eqref{contact}. Further assume that the mappings $v_i$ are  uniformly Lipschitz continuous in the first variable and uniformly contractive in the second variable. 

Then we can associate with the IFS $\cF$ an RB operator $T$ of the form \eqref{eq3.9a} and thus a fractal interpolation problem \eqref{eq3.8} with appropriate compatibility conditions. The attractor $A$ of $\cF$ is then the graph $G(\psi)$ of the solution $\psi$ of \eqref{eq3.8}, respectively, the fixed point of $T$. 
%The commutativity of the diagram \eqref{diagram} then holds with $\cF_T$ replaced by $\cF$ and $T$ replaced by $T_\cF$.
%
\subsection{The case where $\sfF$ is a Banach algebra}
In case that $\sfF$ is in addition a Banach algebra $(\sfF, +, \bullet)$, i.e., when $(\sfF, +)$ is a vector space with a norm $\n{\cdot}$ and a product $\bullet:\sfF\times\sfF$ satisfying
\begin{enumerate}
\item $(\sfF, +, \n{\cdot})$ is a Banach space;
\item $(\sfF, +, \bullet)$ is an associate $\R$-algebra;
\item $\forall\,f_1,f_2\in \sfF: \n{f_1\bullet f_2} \leq \n{f_1}\n{f_2}$,
\end{enumerate}
one obtains a result similar to Theorem \ref{sol}. For the system of functional equations,
\be\label{BA}
\psi (h_i (x)) = q_i (x) + s_i (x)\bullet \psi (x), \quad\text{on $\sfX$ and for $i\in \N_n$},
\ee
define an RB operator on $\cB(\sfX,\sfF)$ by
\be
T f (x) := (q_i\circ h_i^{-1})(x) + (s_i\circ h_i^{-1})(x)\bullet (f\circ h_i^{-1})(x), 
\ee 
for $x\in \sfX_i$ and $i\in \N_n$, where now the functions $q_i, s_i:\sfX\to\sfF$.

In the proof of contractivity, in particular Eqns. \eqref{3.4} and \eqref{3.5}, these estimates now yield
\[
\sup_{x\in\sfX}\n{Tf(x) - Tg(x)}_\sfF \leq \max_{i\in \N_n}\sup\limits_{x\in \sfX} \n{s_i(x)}_\sfF \sup_{x\in\sfX}\n{(f-g)(x)}_\sfF 
\]
Hence, if 
\be\label{eq3.9}
\max_{i\in \N_n}\sup\limits_{x\in \sfX} \n{s_i(x)}_\sfF < 1, 
\ee
then the interpolation problem \eqref{BA} has a unique bounded solution. 
\begin{example}
Consider $\sfF := \sfB(\sfX)$, the Banach algebra of bounded linear operators on $\sfX$ where $\bullet$ is the composition of these operators. Thus, if \eqref{eq3.9} is satisfied then the fractal interpolation problem has a unique $\sfB(\sfX)$-valued solution $\psi:\sfX\to \sfB(\sfX)$:
\[
\psi(x) = q\circ h_i^{-1}(x) + (s_i\circ h_i^{-1})(\psi\circ h_i^{-1})(x),\quad\text{$x\in\sfX_i$ and for $i\in \N_n$},
\]
where we wrote the composition of the operators $s_i\circ h_i^{-1}$ and $\psi\circ h_i^{-1}$ as a juxtaposition.
\end{example}
\subsection{Continuous fractal functions}
Here, we consider continuous fixed points of an RB operator of the form \eqref{eq3.17} in the case of a nonlinear partition of $\sfX$. 

\begin{theorem}\label{th3.3}
The fractal interpolation problem \eqref{psieq} has a unique continuous solution $\psi: \sfX\to\sfF$ provided that
\begin{enumerate}
\item $\sfX = \coprod\limits_{i=1}^n h_i(\sfX)$,
\item the functions $q_i: \sfX\to\sfF$ and $s_i:\sfX\to\R$ are continuous,
\item $s:= \max\limits_{i\in \N_n} \sup\limits_{x\in \sfX} |s_i(x)| < 1$.
\item and for all $i,j\in \N_n$ and $x_1, x_2\in X$ the following join-up conditions are satisfied:
\begin{align}\label{joinupc}
\lim_{x\to x_1} f_j(x) &= f_i(x_2)\\ &\Longrightarrow\;\; \lim_{x\to x_1} q_j(x) + s_j(x) \psi(x) = q_i(x_2) + s_i(x_2) \psi(x_2).\nonumber
\end{align}
\end{enumerate} 
\end{theorem}
\begin{proof}
The proof follows well-established lines and ideas and, therefore, we refer the interested reader to \cite{SB} or \cite{massopust1}. In the former reference, the proof follows the functional equation setting and in the latter the RB operator setting.
\end{proof}

\begin{remark}\label{rem3.1}
Note that by insisting on the validity of the join-up conditions \eqref{joinupc}, we may drop the requirement that $\Pi(h)$ is a \emph{partition} and replace it by the requirement that at the \emph{contact points} 
\be\label{contact}
\{x_1\in \sfX : \exists x_2\in \sfX, \exists i_1, i_2\in\N_n, i_1\neq i_2, h_{i_1}(x_1) = h_{i_2}(x_2)\}
\ee
\eqref{joinupc} holds. See also \cite{SB}. 

In the following, we will tacitly make use of this observation and abuse notation by still writing $\sfX = \coprod\limits_{i=1}^n h_i(\sfX)$.
\end{remark}

An example of a continuous solution $\psi$ of the fractal interpolation problem is shown below. This example also demonstrates the affect of a nonlinear partition on the graph of a fractal function.

\begin{example}\label{ex3.4}
Consider the set up of Examples \ref{ex2.1}, let $\sfX := [0,1]$, and let $\sfF:= \R$. Note that the maps $h_1, h_2$ do not generate a partition of $[0,1]$ but at the contact point $x=\frac12$ they have the same value. (See Remark \ref{rem3.1}.) The inverse maps for $h_1$ and $h_2$ are given by
\[
h_1^{-1}: [0,\tfrac12]\to [0,1],\quad x\mapsto  -1 +\sqrt{6x+1}
\]
and
\[ 
h_2^{-1}: [\tfrac12,1]\to [0,1], \quad x\mapsto \tfrac4\pi \arcsin\tfrac{2x-1}{\sqrt{2}},
\]
respectively. Define an RB operator $T: \cC(\sfX,\R)\to \cC(\sfX,\R)$ by
\[
T f (x) := \begin{cases}
-1 +\sqrt{6x+1} + \tfrac12 f(-1 +\sqrt{6x+1}), & x\in [0,\tfrac12),\\
1 - \tfrac4\pi \arcsin\tfrac{2x-1}{\sqrt{2}} + \tfrac12 f(\tfrac4\pi \arcsin\tfrac{2x-1}{\sqrt{2}}), & x\in [\tfrac12, 1).
\end{cases}
\]
Here, $\cC(\sfX) := \cC(\sfX,\R)$ denotes the Banach space of continuous function $\sfX\to \R$ endowed with the supremum norm. Note that the compatibility condition \eqref{joinupc} at the contact point is satisfied.

In Figure \ref{fig1}, the graph of the fixed point $\psi$ and the Takagi function are shown. The latter uses an affine partition generated by $x\mapsto \frac12 x$ and $x\mapsto \frac12 (x+1)$ of $\sfX$.
\begin{figure}[h!]
\begin{center}
\includegraphics[width= 5cm, height = 3.5cm]{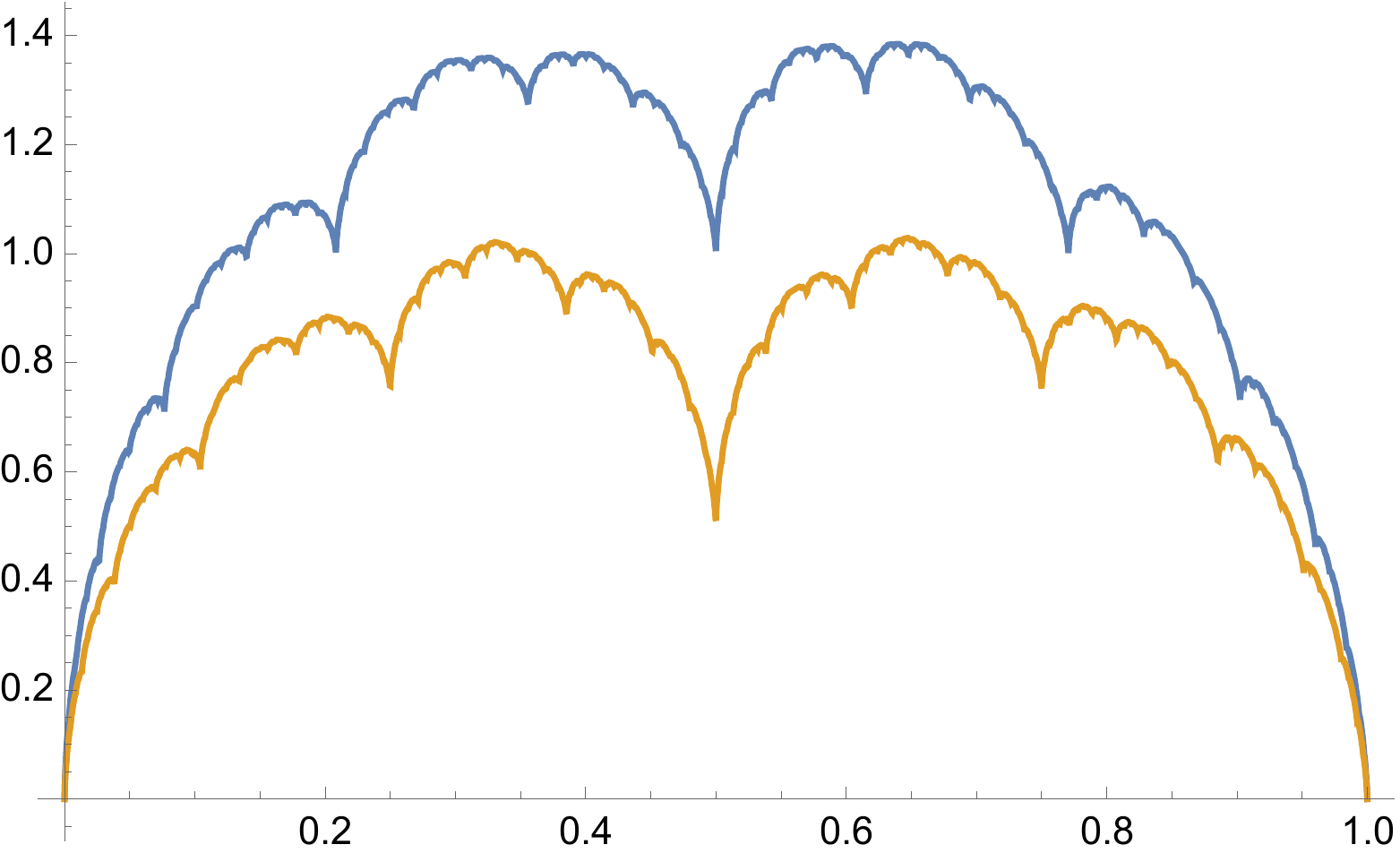}
\caption{The graph of $\psi$ (upper graph) and the graph of the Takagi function (lower graph).}\label{fig1}
\end{center}
\end{figure}
\end{example}
Now, consider nonlinear mappings $h_i(\cdot, \vep)$ depending on a parameter $\vep\in (0,1]$ in such a way that for $\vep = 0$, $h_i (\cdot, 0)$ is an affine mapping. The example below illustrates this approach.
\begin{example}\label{ex3.5}
Let again $\sfE := \R$, $\sfX := [0,1]$, and let $\vep\in [-\frac12, \frac12]\subset\R$ be fixed. Define nonlinear mappings 
\[
h_1(\cdot, \vep): [0,1]\to [0,\tfrac12],\quad x\mapsto (\tfrac12 - \vep) x + \vep x^2
\]
and
\[ 
h_2(\cdot, \vep): [0,1]\to [\tfrac12,1],\quad x\mapsto := \tfrac12 + (\tfrac12+\vep) x - \vep x^2.
\]
Clearly, $h_1(x,0) = \frac12 x$ and $h_2(x,0) = \frac12(1+x)$ are an affine partition of $\sfX$ with contact point $x=\frac12$ (see Example \ref{ex3.4}). The inverse functions $h_1 (\cdot, \vep)$ and $h_2 (\cdot, \vep)$ are given by
\[
h_1^{-1} (x, \vep) = \frac{-1+2\vep+\sqrt{1+4 (4 x-1) \vep +4 \vep ^2}}{4 \vep }
\]
and
\[
h_2^{-1} (x, \vep) = \frac{1+2\vep+\sqrt{1+ 4(3 - 4x)\vep +4 \vep ^2}}{4 \varepsilon },
\]
respectively. Note that $\lim\limits_{\vep\downarrow 0} h_i^{-1} (x,\vep)$ produces $h_i^{-1} (x)$ and that the inverses $h_i^{-1}$ only exist for $\vep\in [-\frac12, \frac12]$.

Define an RB operator $T (\vep): \cC(\sfX)\to \cC(\sfX)$ by
\[
T(\vep) f (x) := \begin{cases}
h_1^{-1}(x,\vep) + \tfrac12 f(h_1^{-1}(x,\vep)), & x\in [0,\tfrac12),\\
h_2^{-1}(x,\vep) + \tfrac12 f(h_2^{-1}(x,\vep)), & x\in [\tfrac12, 1).
\end{cases}
\]
The fixed point $\psi_\vep$ for various values of $\vep$ is depicted in Figure \ref{fig2}.
\begin{figure}[h!]
\begin{center}
\includegraphics[width= 6cm, height = 4cm]{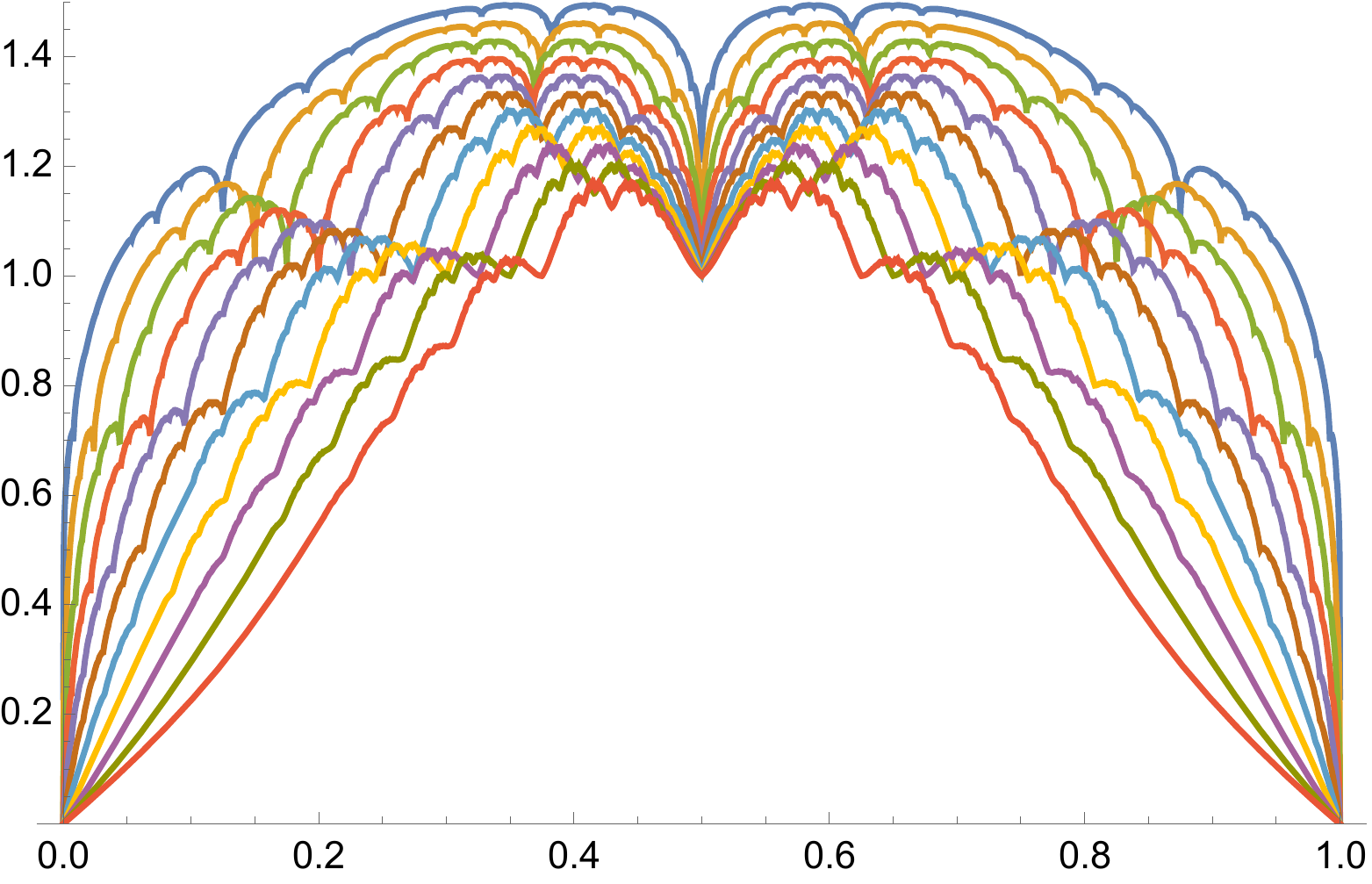}
\caption{The graphs of $\psi_\vep$ for $\vep = \frac{j}{10}$, $j \in \{-5, \ldots, 5\}$ (top to bottom).}\label{fig2}
\end{center}
\end{figure}
\end{example}
\subsection{Some perturbation theory}
For the purposes of this subsection, we require the following result. A similar result can also be found in reference \cite{barnsley}.
\begin{proposition}\label{prop3.1}
Let $(\sfE, \n{\cdot})$ be a Banach space and let $(\sfP, \n{\cdot}_\sfP)$ be a normed space. Suppose that $T(p): \sfE\to\sfE$ is a contraction for each $p\in\sfP$ with Lipschitz constant $L_p$. Assume that there exists an $L\in[0,1)$ such that $L_p \leq L$, for all $p\in \sfP$. Further assume that there exists a $p_0\in \sfP$ such that $\lim\limits_{p\to p_0} T(p) (e) = T(p_0) (e)$, for all $e\in \sfE$. Then the fixed point equation $T(p) e = e$ possesses for each $p\in \sfP$ a unique solution $e_p\in \sfE$ and $\lim\limits_{p\to p_0} e_p = e_{p_0}$.
\end{proposition}
\begin{proof}
The Banach Fixed Point Theorem guarantees a unique solution $e_p$ of the fixed point equation for each $p\in \sfP$. Furthermore,
\begin{align*}
\n{e_p - e_{p_0}} & = \n{T(p) e_p - T(p_0) e_{p_0}}\\
&\leq \n{T(p) e_p - T(p) e_{p_0}} + \n{T(p) e_{p_0} - T(p_0) e_{p_0}}\\
& \leq L \n{e_p - e_{p_0}} + \n{T(p) e_{p_0} - T(p_0)e_{p_0}},
\end{align*}
and, therefore, 
\[
\n{e_p - e_{p_0}} \leq \frac{1}{1-L}\, \n{T(p) e_{p_0} - T(p_0) e_{p_0}}.
\]
The statement now follows by taking the limit $p\to p_0$.
\end{proof}

Now, consider the case $(\sfE,\n{\cdot}_\sfE):= (\cC(\sfX), \n{\cdot})$, where $\n{\cdot}$ denotes the supremum norm, and $(\sfP, \n{\cdot}_\sfP) := (\R, \abs{\cdot})$. Assume that the functions $h_i : \sfX\times \R\to\sfX$, $(x,\vep)\mapsto h_i(x,\vep)$ and their inverses $h_i^{-1}$ depend continuously on the parameter $\vep\in I_1$, for some nonempty interval $I_1\subset\R$, and generate a partition of $\sfX$. (See also Remark \ref{rem3.1} in this context.)

Define an RB operator $T(\vep) : \cC(\sfX,\sfF)\to\cC(\sfX,\sfF)$ by
\be\label{Tvep}
T(\vep) f (x) := q_i (h_i^{-1}(x,\vep)) + s_i (h_i^{-1}(x,\vep))\cdot f(h_i^{-1}(x,\vep)), 
\ee
for $x\in \sfX_i$ and $i\in \N_n$, where the functions $q_i: \sfX\to\sfF$ and $s_i:\sfX\to\R$ are continuous.

Suppose there exists a nonempty interval $I_2\subset\R$ such that 
\[
s_\vep:= \max\limits_{i\in \N_n} \sup\limits_{x\in \sfX_i} |s_i(h_i^{-1}(x,\vep))| \leq s < 1, \quad\forall \vep\in I_2,
\]
for some $s\geq 0$. Then, if $I:= I_1\cap I_2\neq\emptyset$, $T(\vep)$ is a contraction on $\cC(\sfX,\sfF)$ for all $\vep\in I$. Proposition \ref{prop3.1} thus implies that the fixed point $\psi_\vep\in\cC(\sfX,\sfF)$ of $T(\vep)$ is continuous in $\vep$ for all $\vep\in I$. 

We summarize these findings in the next theorem.
\begin{theorem}
Suppose an RB operator $T(\vep)$ is defined as in $\eqref{Tvep}$. Suppose further that there exists a nonempty interval $I\subset\R$ such that 
\begin{enumerate}
\item[(i)] the functions $h_i : \sfX\times \R\to\sfX$, $(x,\vep)\mapsto h_i(x,\vep)$, and their inverses $h_i^{-1}$ depend continuously on $\vep\in I$ and 
\item[(ii)] $s_\vep:= \max\limits_{i\in \N_n} \sup\limits_{x\in \sfX_i} |s_i(h_i^{-1}(x,\vep))| \leq s < 1$, for all $\vep\in I$ and some $s\geq 0$. 
\end{enumerate}
Then the fixed point $\psi_\vep\in\cC(\sfX,\sfF)$ of $T(\vep)$, respectively, the solution of the underlying fractal interpolation problem depending on the parameter $\vep$,
\[
\psi_\vep =  q_i (h_i^{-1}(x,\vep)) + s_i (h_i^{-1}(x,\vep))\cdot \psi_\vep(h_i^{-1}(x,\vep)),\quad x\in \sfX_i\;i\in \N_n,
\]
is continuous in $\vep$ for all $\vep\in I$. 
\end{theorem}
We remark, that in case $s_i \in\R$, for all $i\in\N_n$, and $\max\limits_{i\in\N_n} \abs{s_i} =: s < 1$, the solution of the fractal interpolation problem, respectively, the fixed point $\psi_\vep$ is continuous in $\vep$, for all $\vep\in I_1$. This situation is illustrated in Example \ref{ex3.5} where $I_1 = [-\frac12,\frac12]$.
\section{Bochner-Lebesgue Spaces}
To this end, assume again that $\sfX$ is a nonempty bounded subset of a Banach space $\sfE$ and that $\X :=(\sfX,\sA, \mu)$ is a complete $\sigma$-finite measure space with $\sigma$-algebra $\sA$ and scalar-valued measure $\mu$. Recall that for $1\leq p\leq \infty$, the Bochner-Lebesgue spaces $L^p (\X,\sfF)$ consist of (equivalence classes of) Bochner measurable functions $f:\sfX\to \sfF$ such that
\[
\|f\|_{L^p(\X,\sfF)} := \left(\int_{\sfX} \|f(x)\|_\sfF^p \,d\mu(x)\right)^{1/p} < \infty, \quad 1 \leq p < \infty,
\]
and
\[
\|f\|_{L^\infty(\X,\sfF)} := \esssup_{x\in \sfX} \|f(x)\|_\sfF < \infty, \quad p = \infty.
\]
For $0 < p <1$, the spaces $L^p(\X,\sfF)$ are defined as above but instead of a norm, a metric is used to obtain completeness. More precisely, for $p\in (0,1)$, define $d_p : L^p(\X, \sfF)\times L^p(\X,\sfF)\to \R$ by
\[
d_p (f,g) := \|f - g\|_{\sfF}^p.
\]
Note that the inequality $(a+b)^p \leq a^p + b^p$ holds for all $a,b\geq 0$. Then, $(L^p(\X,\sfF), d_p)$ becomes an $F$-space, i.e., a topological vector space whose topology is induced by a complete translation-invariant metric \cite{R}.  For more details about these spaces and related topics mentioned above, we refer the reader to, for instance, \cite{A,HNVW,rudin}.\\

%In the following, we set $\sfX\subset\sfE:=\R^m$ and $\sfY:=\R^k$ where the Euclidean spaces $\R^m$ and $\R^k$ are endowed with the corresponding canonical Euclidean norms.
%
%Recall that the (real) Lebesgue spaces $L^p (\sfX, \R^k)$, where $\sfX\subset \R^m$ is nonempty, are defined as consisting of (equivalence classes of) functions $f:\sfX\to\R^k$ for which
%\[
%\n{f}_p :=\begin{cases}\left(\displaystyle{\int_\sfX \n{f(x)}^p dx} \right)^{1/p}, & 1\leq p < \infty;\\ \\
%\esssup\limits_{x\in\sfX} \n{f(x)}, & p = \infty.
%\end{cases}
%\]
%is finite. Here, $\n{f(x)} := \sqrt{\abs{f_1(x)}^2 + \cdots + \abs{f_k(x)}^2}$ with $f := (f_1, \ldots, f_k)$.
In order to solve the fractal interpolation problem \eqref{psieq} for Bochner- Lebesgues spaces, conditions on the functions $q_i$ and $s_i$ need to be found such that the solution $\psi$ is an element of $L^p(\X, \sfF)$, for $0 < p \leq \infty$. 

To this end, we assume that $\sfE := \R^m$ in which case we take for $\mu$ Lebesgue measure $\lambda$. Note that in order for $\psi$ to be in $L^p(\X, \sfF)$, the RB operator $T$ must map $L^p(\X, \sfF)$ into itself. Therefore, the functions $q_i$ and $s_i$ must also be in $L^p(\X, \sfF)$. Moreover, $s_i$ needs to be in $L^\infty(\X, \sfF)$ for the pointwise product $s_i\cdot f$ to be in $L^p(\X, \sfF)$. Thus, as $\sfX$ is bounded it has finite measure and therefore $s_i\in L^\infty(\X, \R)$ implies that $s_i\in L^p(\X, \R)$ for all $1\leq p \leq \infty$.

Now, it remains to be established under what conditions the RB operator $T$ is contractive on $L^p(\X, \sfF)$. For this purpose, let $f,g\in L^p(\X, \sfF)$. Then, with $\sfX_i:= h_i(\sfX)$ and for $1\leq p \leq\infty$,
\begin{align*}
\n{Tf - Tg}_{L^p(\X,\sfF}^p &= \int_\sfX \n{Tf (x) - Tg(x)}_\sfF^p d\lambda(x)\\
&= \int_\sfX \n{\sum_{i=1}^n s_i(h_i^{-1}(x))\cdot (f-g)(h_i^{-1}(x)) \one_{\sfX_i}(x)}_\sfF^p d\lambda(x)\\
&\leq \sum_{i=1}^n \int_{\sfX_i} \abs{s_i(h_i^{-1}(x))}^p \n{(f-g)(h_i^{-1}(x))}_\sfF^p d\lambda(x)\\
&= \sum_{i=1}^n \int_{\sfX} \abs{\det{(D h_i)(x)}} \abs{s_i(x)}^p \n{(f-g)((x)}^p_\sfF d\lambda(x),
\end{align*}
where the transformation theorem for Lebesgue measure was used and $D$ denotes the derivative in $\R^m$.

Let 
\be
\gamma := \sum_{i=1}^n\, \sup_{x\in \sfX}\left\{
\abs{\det{(D h_i)(x)}} \abs{s_i(x)}^p
\right\}.
\ee
If $\gamma < 1$, then 
\[
\n{Tf - Tg}_p \leq \gamma \n{f-g}_p
\]
and the operator $T$ is contractive on $L^p (\X, \sfF)$.

Hence, we arrived at the following result
\begin{theorem}\label{th3.2}
The system of functional equations, respectively, the fractal interpolation problem,
\[
\psi (h_i (x)) = q_i (x) + s_i (x) \psi (x), \quad\text{on $\sfX\subset\R^m$ and for $i\in \N_n$},
\]
has a unique solution $\psi\in L^p(\X, \sfF)$, $0\leq p \leq \infty$, respectively, the RB operator
\[
T f (x) = (q_i\circ h_i^{-1})(x) + (s_i\circ h_i^{-1})(x)\cdot (f\circ h_i^{-1})(x),\quad{x\in\sfX_i},\;i\in \N_n,
\]
a unique fixed point $\psi\in L^p(\X, \sfF)$ provided that
\begin{enumerate}
\item $q_i\in L^p(\X, \sfF)$, $s_i\in L^\infty (\X, \sfF)$ and 
\item 
\[
\gamma = \begin{cases}
\displaystyle{\sum\limits_{i=1}^n}\, \sup\limits_{x\in \sfX}\left\{\abs{\det{(D h_i)(x)}} \abs{s_i(x)}^p\right\} , & 0 < p < \infty;\\
\max\limits_{i\in \N_n}\sup\limits_{x\in \sfX}\left\{\abs{s_i(x)}^p\right\}, & p = \infty;
\end{cases}\quad < 1.
\]
\end{enumerate}
\end{theorem}
\begin{proof}
The proof of contractivity for $T$ in the cases $0< p<1$ and $p=\infty$ proceeds along the same lines as above and is omitted.
\end{proof}
\section{$\cC^\alpha$ Spaces}
In this section, we derive a sufficient condition for the global function $\psi$ to be an element of a $\cC^\alpha$ function space. To this end, let $\sfU$ be a nonempty open subset of $\sfE$ with compact closure $\oU$. Following \cite{A}, we define the vector space $\cC^\alpha(\oU,\sfF)$ to consist of all those functions $f\in \cC^k(\sfU,\sfF)$ for which $D^k f$ is bounded and continuous on $\sfU$ for $k\in \{0,1,\ldots, \alpha\}$ (hence possesses a unique bounded continuous extension to $\oU$). Here, $D^k$ denotes the $k$-fold Fr\'echet derivative of $f$. 

In the following, we denote the closure $\oU$ of $\sfU$ by $\sfX$. The vector space $\cC^\alpha(\sfX,\sfF)$ becomes a Banach space when endowed with the norm
\be\label{Calpha}
\n{f}_{C^\alpha} := \max_{k=0,1,\ldots, \alpha} \n{D^k f}.
\ee
(See also Remark \ref{rem5.1} below.)

We introduce an RB operator $T: \cC^\alpha(\sfX,\sfF)\to \sfF^\sfX$ in the form \eqref{eq3.17} but now insist that the functions $q_i\in \cC^\alpha (\sfX,\sfF)$ and $s_i\in \cC^\alpha(\sfX,\R)$, $i\in \N_n$. Moreover, we require for all $i,j\in \N_n$ with $x_1, x_2\in X$ and all $k\in \{0,1,\ldots, \alpha\}$ the following join-up conditions are satisfied:
\begin{align}\label{Calphajoinup}
\lim_{x\to x_1} & D^k f_j(x) = D^kf_i(x_2)\;\;\Longrightarrow\\
& \lim_{x\to x_1} D^k q_j(x) + D^k(s_j\circ\psi)(x) = D^k q_i(x_2) + D^k(s_i\circ \psi)(x_2).\nonumber
\end{align}
Under these assumptions, $T$ is well-defined and maps $\cC^\alpha(\sfX,\sfF)$ into itself. We observe that Remark \ref{rem3.1} also applies to the above join-up conditions at appropriately defined contact points.

The main issue here will be the derivation of conditions under which the RB operator acting on $\cC^\alpha(\sfX,\sfF)$ is contractive. To this end, it suffices to show contractivity of $T$ for a single $f\in\cC^\alpha(\sfX,\sfF)$. We compute on $\sfX_i$
\begin{align*}
\n{Tf}_{C^\alpha} & = \max_{k=0,1,\ldots, \alpha}\ \max_{i=1,\ldots, n}\ \n{D^k (s_i\circ h_i^{-1}\cdot f\circ h_i^{-1})}\\
&\leq  \max_{k=0,1,\ldots, \alpha}\ \max_{i=1,\ldots, n}\ \sum_{q=0}^k \binom{k}{q} \n{D^{k-q} (s_i\circ h_i^{-1})} \n{D^q ( f\circ h_i^{-1})},
\end{align*}
by the Leibnitz formula in Banach spaces \cite[§1.3]{AR} and the fact that $\cC^\alpha(\sfX,\sfF)$ is a Banach algebra. It remains to estimate the terms $\n{D^q ( f\circ h_i^{-1})}$. 

For this purpose, the application of Fa\`a di Burno's formula in Banach spaces \cite[§1.4]{AR} produces 
\begin{align*}
D^q (f\circ h_i^{-1})(x) &= \sum_{1\leq r\leq q} \sum_{\substack{(i_1, \ldots, i_r)\in \N^r\\ i_1+\cdots i_r = q}} \sigma_q\, D^r f (h_i^{-1}(x))\, D^{i_1} h_i^{-1}(x)\cdot\ldots\cdot D^{i_r} h_i^{-1}(x),
\end{align*}
where the $\sigma_q = \sigma_q (i_1, \ldots, i_r)$ are nonnegative combinatorial constants defined inductively by
\begin{gather*}
\sigma_1 (1) = 1,\\
\sigma_{q+1}(q+1) = 1, \quad\text{and}\quad \sigma_{q+1}(i_1, \ldots, i_r,q+1-r) = \binom{q}{r} \sigma_q(i_1, \ldots, i_r).
\end{gather*}
Hence,
\[
\n{D^q ( f\circ h_i^{-1})} \leq \left(\sum_{1\leq r\leq q} \sum_{\substack{(i_1, \ldots, i_r)\in \N^r\\ i_1+\cdots i_r = q}} \sigma_q\,\n{D^{i_1} h_i^{-1}(x)\cdot\ldots\cdot D^{i_r} h_i^{-1}}\right)\n{f}_{C^\alpha}
\]
Combining the above expressions yields
\begin{align*}
\n{Tf}_{C^\alpha} &\leq \max_{k=0,1,\ldots, \alpha}\ \max_{i=1,\ldots, n}\ \sum_{0\leq q\leq k} \sum_{1\leq r\leq q} \sum_{\substack{(i_1, \ldots, i_r)\in \N^r\\ i_1+\cdots i_r = q}}  \binom{k}{q} \sigma_q \n{D^{k-q} (s_i\circ h_i^{-1})} \\
& \qquad\qquad\times\n{D^{i_1} h_i^{-1}(x)\cdot\ldots\cdot D^{i_r} h_i^{-1}}\,\n{f}_{C^\alpha}
\end{align*}
Thus, $T$ is a contraction on $\cC^\alpha(\sfX,\sfF)$ provided that
\be\label{gamma}
\gamma := \max_{k=0,1,\ldots, \alpha}\ \max_{i=1,\ldots, n}\ \sum_{0\leq q\leq k} \sum_{1\leq r\leq q} \sum_{\substack{(i_1, \ldots, i_r)\in \N^r\\ i_1+\cdots i_r = q}}  \binom{k}{q} \sigma_q\, \gamma_s (i,k,q)\, \gamma_h (i,r) < 1,
\ee
where for simplicity we set
\[
\gamma_s (i, k, q) := \n{D^{k-q} (s_i\circ h_i^{-1})}
\]
and
\[
\gamma_h (i, r) := \n{D^{i_1} h_i^{-1}(x)\cdot\ldots\cdot D^{i_r} h_i^{-1}}.
\]
\begin{theorem}\label{th5.1}
Let $\gamma$ be defined as in \eqref{gamma}. Assume that $q_i\in \cC^\alpha (\sfX,\sfF)$ and $s_i\in \cC^\alpha(\sfX,\R)$, $i\in \N_n$, and the join-up conditions \eqref{Calphajoinup} are satisfied. Then, the fractal interpolation problem \eqref{psieq} has a unique solution in $\cC^\alpha (\sfX,\sfF)$ provided that $\gamma < 1$.
\end{theorem}
\begin{remark}\label{rem5.1}
Using an equivalent norm to define the norm topology on $\cC^\alpha(\sfX,\sfF)$ such as 
\[
\n{f}_{C^\alpha} := \sum_{k=0}^\alpha \n{D^k f}
\]
does not invalidate the result stated in Theorem \ref{th5.1}. This follows directly from Lemmas \ref{lem1} and \ref{lem2} below.
\begin{lemma}\label{lem1}
Suppose $\n{\cdot}_1$ and $\n{\cdot}_2$ are equivalent norms on a Banach space $\sfE$, i.e., there exist positive constants $c_1\leq c_2$ such that
\[
c_1 \n{\cdot}_1 \leq \n{\cdot}_2 \leq c_2\n{\cdot}_1.
\]
It $T$ is a contraction on $\sfE$ with respect to $\n{\cdot}_1$ then there exists an $m_0\in \N$ such that $T^{m_0}$ is a contraction with respect to $\n{\cdot}_2$.
\end{lemma}
\begin{proof}
Suppose that there exists a $\gamma\in [0,1)$ such that
\[
\n{Te}_1 \leq \gamma\n{e}_1, \quad\forall\,e\in \sfE.
\]
Then
\begin{align*}
\n{Te}_2 &\leq c_2 \n{Te}_1 \leq c_2 \gamma \n{e}_1 \leq \left(\frac{c_2}{c_1}\,\gamma\right)\n{e}_2.
\end{align*}
As $\n{T^m f}_1 \leq \gamma^m\n{e}_1$, for all $m\in \N$, there exists an $m_0\in\N$ so that $\frac{c_2}{c_1}\,\gamma^{m_0} < 1$.
\end{proof}
\begin{lemma}\label{lem2}
Suppose that for some $m\in \N$, the mapping $T^m:\sfE\to\sfE$ is a contraction and thus has a unique fixed point $e^*$. Then $e^*$ is also the unique fixed point of $T$.
\end{lemma}
\begin{proof}
See, for instance, \cite[Corollary 1.3]{P}.
\end{proof}
\end{remark}
\begin{example}
Take $\sfX:=[0,1]\subset\R =: \sfE$ and $\sfF := \R$ and consider $\alpha:= 1$. Furthermore, let
\begin{gather*}
h_1 (x) := \tfrac14 x(x+1)\quad\text{and}\quad h_2(x) := \tfrac14(2+x+x^2),\\
s_1(x) := \tfrac15(1-x)\quad\text{and}\quad s_2(x):= \tfrac15 x,\\
q_1(x) := x\quad\text{and}\quad q_2(x):= 1+\tfrac{3}{13}x.
\end{gather*}
One easily verifies that the requirements on $h_i$, $s_i$, and $q_i$, as well as the join-up conditions are satisfied, and that $\gamma = \frac45 < 1$. Hence, the solution of the fractal interpolation problem \eqref{psieq} in the function space $\cC([0,1],\R)$ is an element of the function space. Figure \ref{fig4} below displays the graph of $\psi$.
\begin{figure}[h!]
\begin{center}
\includegraphics[width= 5cm, height = 4cm]{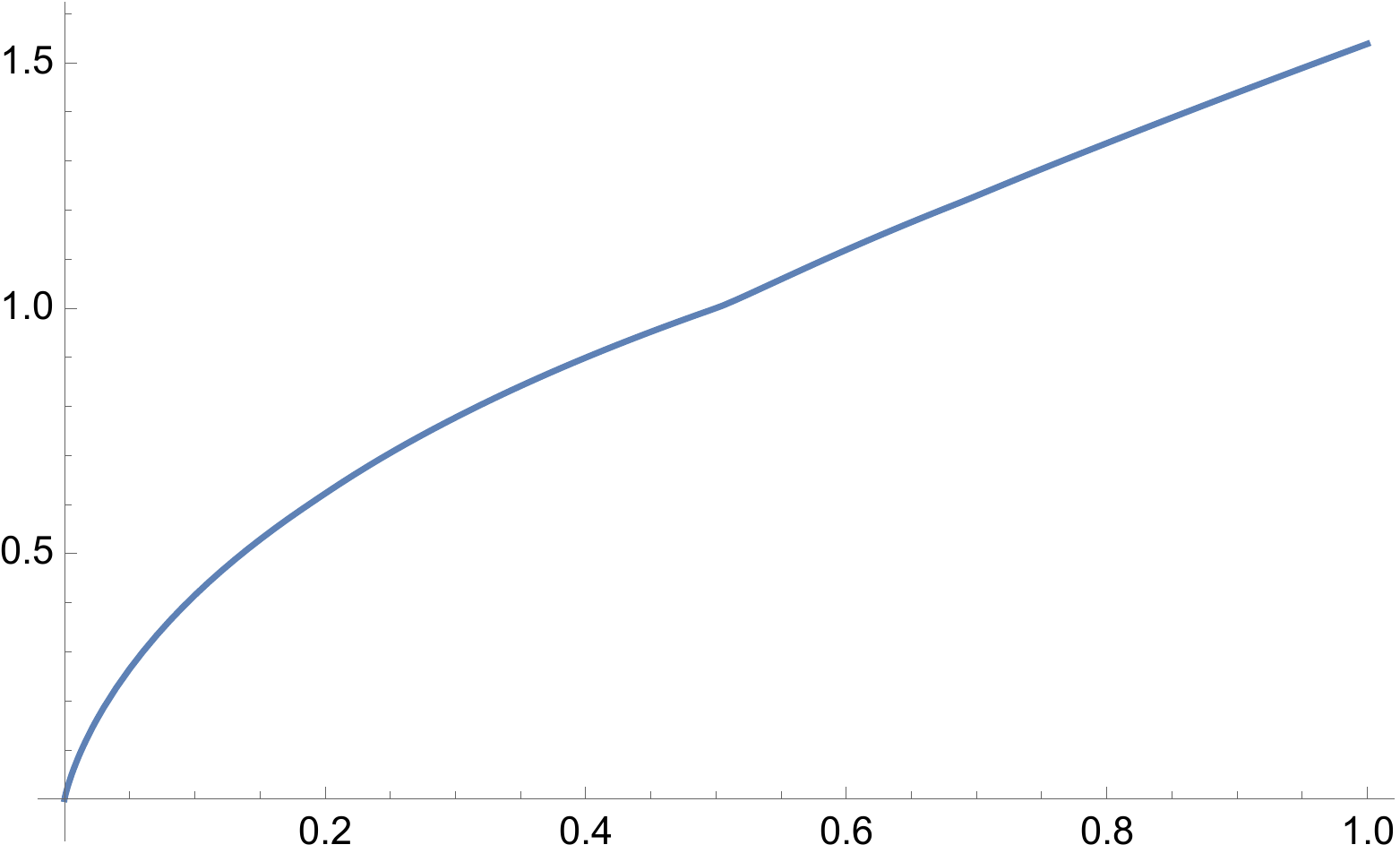}
\caption{The graph of a differentiable fractal function.}\label{fig4}
\end{center}
\end{figure}
\end{example}

\end{document}